\documentclass[a4paper]{amsart}
\usepackage[latin9]{inputenc}
\usepackage{amsthm}
\usepackage{amssymb}
\usepackage{esint}
\usepackage[numbers]{natbib}

\makeatletter

\pdfpageheight\paperheight
\pdfpagewidth\paperwidth

\numberwithin{equation}{section}
\numberwithin{figure}{section}
\theoremstyle{plain}
\newtheorem{thm}{\protect\theoremname}[section]
  \theoremstyle{plain}
  \newtheorem{lem}[thm]{\protect\lemmaname}
  \theoremstyle{plain}
  \newtheorem{prop}[thm]{\protect\propositionname}
  \theoremstyle{remark}
  \newtheorem*{rem*}{\protect\remarkname}
  \theoremstyle{definition}
  \newtheorem{example}[thm]{\protect\examplename}

\usepackage{amsthm}
\usepackage{graphicx}

\newtheorem*{thmA*}{Theorem}
\DeclareMathOperator{\SC}{SC}

\makeatother

  \providecommand{\examplename}{Example}
  \providecommand{\lemmaname}{Lemma}
  \providecommand{\propositionname}{Proposition}
  \providecommand{\remarkname}{Remark}
\providecommand{\theoremname}{Theorem}

\begin{document}

\title[The quantization for Markov-type measures]{The quantization for Markov-type measures on a class of ratio-specified
graph directed fractals}

\author{Marc Kesseb\"ohmer}

\address{Marc Kesseb\"ohmer, Fachbereich 3 -- Mathematik und Informatik,
Universität Bremen, Bibliothekstr. 1, Bremen 28359, Germany}

\email{mhk@math.uni-bremen.de}

\author{Sanguo Zhu}

\address{Sanguo Zhu, School of Mathematics and Physics, Jiangsu University
of Technology, Changzhou 213001, China}

\email{sgzhu@jsut.edu.cn {\rm(Corresponding author)}}

\thanks{The second author was supported by the China Schoolarship Council,
File No. 201308320049.}
\begin{abstract}
We study the asymptotic quantization error of order $r$ for Markov-type
measures $\mu$ on a class of ratio-specified graph directed fractals.
We show that the quantization dimension of $\mu$ exists and determine
its exact value $s_{r}$ in terms of spectral radius of a related
matrix. We prove that the $s_{r}$-dimensional lower quantization
coefficient of $\mu$ is always positive. Moreover, inspired by Mauldin-Williams's
work on the Hausdorff measure of graph directed fractals, we establish
a necessary and sufficient condition for the $s_{r}$-dimensional
upper quantization coefficient of $\mu$ to be finite. 
\end{abstract}



\maketitle

\section{Introduction}

In this paper, we study the asymptotics of the quantization error
for Markov-type measures on a class of ratio-specified graph directed
fractals. One of the main mathematical aims of the quantization problem
is to study the error in the approximation of a given probability
measure with probability measures of finite support. We refer to \citep{GL:00,GL:01,GL:04,GL:12,PK:01}
for more theoretical results and \citep{PG:98,PG:05} for promising
applications of quantization theory. One may see \citep{GN:98,Za:63}
for its deep background in information theory and engineering technology.
In the following, let us recall some of the crucial definitions and
known results.

\subsection{The upper (lower) quantization dimension and quantization coefficient}

We set $\mathcal{D}_{n}:=\{\alpha\subset\mathbb{R}^{q}:1\leq{\rm card}(\alpha)\leq n\}$
for $n\in\mathbb{N}$. Let $\nu$ be a Borel probability measure on
$\mathbb{R}^{q}$. For every $n\geq1$, the $n$th quantization error
for $\nu$ of order $r$ is defined by (see \citep{GL:00} for a number
of equivalent definitions): 
\begin{eqnarray}
e_{n,r}(\nu):=\inf_{\alpha\in\mathcal{D}_{n}}\bigg(\int d(x,\alpha)^{r}d\nu(x)\bigg)^{1/r}.\label{quanerrordef}
\end{eqnarray}
Here $d(x,\alpha):=\inf_{a\in\alpha}d(x,a)$ and $d(\cdot,\cdot)$
is the metric induced by an arbitrary norm on $\mathbb{R}^{q}$. For
$r\geq1$, $e_{n,r}(\nu)$ agrees with the error in the approximation
of $\nu$ by discrete probability measures supported on at most $n$
points, in the sense of the Wasserstein $L_{r}$-metric \citep{GL:00}.

The convergence rate of $e_{n,r}(\nu)$ is characterized by the upper
and lower quantization dimension of order $r$ as defined below: 
\begin{eqnarray*}
\overline{D}_{r}(\nu):=\limsup_{n\to\infty}\frac{\log n}{-\log e_{n,r}(\nu)},\;\underline{D}_{r}(\nu):=\liminf_{n\to\infty}\frac{\log n}{-\log e_{n,r}(\nu)}.
\end{eqnarray*}
If $\overline{D}_{r}(\nu)=\underline{D}_{r}(\nu)$, the common value
is denoted by $D_{r}(\nu)$ and called the quantization dimension
for $\nu$ of order $r$. For $s>0$, we define the $s$-dimensional
upper and lower quantization coefficient for $\nu$ of order $r$
by (cf. \citep{GL:00,PK:01}) 
\[
\overline{Q}_{r}^{s}(\nu):=\limsup_{n\to\infty}n^{r/s}e_{n,r}(\nu)^{r},\;\;\underline{Q}_{r}^{s}(\nu):=\liminf_{n\to\infty}n^{r/s}e_{n,r}(\nu)^{r}.
\]
According to \citep[Proposition 11.3]{GL:00} (see also \citep{PK:01}),
the upper (lower) quantization dimension is exactly the critical point
at which the upper (lower) quantization coefficient jumps from zero
to infinity. Compared with the dimensions, the coefficients provide
us with more accurate information for the asymptotic properties of
the quantization error whenever they are both positive and finite.
Therefore, it is one of the standard topics in the quantization problem
to examine the finiteness and positivity of the latter. So far, the
upper (lower) quantization coefficient has been well studied for absolutely
continuous probability measures \citep[Theorem 6.2]{GL:00} and some
classes of fractal measures, such as self-similar measures \citep{GL:01,GL:02}
and diadic homogeneous Cantor measures \citep{Kr:08}.

Let $(f_{i})_{i=1}^{N}$ be a family of contractive similitudes on
$\mathbb{R}^{q}$. By \citep{Hut:81}, there exists a unique Borel
probability measure $\nu$ satisfying $\nu=\sum_{i=1}^{N}q_{i}\nu\circ f_{i}^{-1}$
associated with $(f_{i})_{i=1}^{N}$ and a probability vector $(q_{i})_{i=1}^{N}$.
Recall that $(f_{i})_{i=1}^{N}$ is said to satisfy the open set condition
(OSC), if there exists a non-empty bounded open set $U$ such that
$\bigcup_{i=1}^{N}f_{i}(U)\subset U$ and $f_{i}(U)\cap f_{j}(U)=\emptyset$
for all $1\leq i\neq j\leq N$. Graf and Luschgy proved the following
result which often provides us with significant insight into the study
for non-self-similar probability measures: 

\begin{thmA*}[Graf/Luschgy \citep{GL:01,GL:02}] Assume that $(f_{i})_{i=1}^{N}$
satisfies the OSC. Let $\nu$ be the self-similar measure associated
with $(f_{i})_{i=1}^{N}$ and a probability vector $(q_{i})_{i=1}^{N}$.
Let $k_{r}$ be the unique solution of the equation $\sum_{i=1}^{N}(q_{i}s_{i}^{r})^{\frac{k_{r}}{k_{r}+r}}=1$.
Then 
\[
D_{r}(\nu)=k_{r},\;\;0<\underline{Q}_{r}^{k_{r}}(\nu)\leq\overline{Q}_{r}^{k_{r}}(\nu)<\infty.
\]
\end{thmA*}

\subsection{A class of graph directed fractals and Markov-type measures}

In this subsection, we recall the definitions of a class of graph
directed fractals and Markov-type measures on these sets. One may
see \citep{BED:89,EM:92,MW:88} for more details.

Let $P:=(p_{ij})_{N\times N}$ be a row-stochastic matrix, namely,
$p_{ij}\geq0,1\leq i,j\leq N$, and $\sum_{j=1}^{N}p_{ij}=1,1\leq i\leq N$.
We always assume 
\begin{equation}
{\rm card}(\{1\leq j\leq N:p_{ij}>0\})\geq2\;\;{\rm for\; all}\;\;1\leq i\leq N.\label{cardpij>0}
\end{equation}
We will need the following notations. Set 
\begin{eqnarray*}
 &  & \theta:={\rm empty\; word},\;\;\Omega_{0}:=\{\theta\},\;\Omega_{1}:=\{1,\ldots N\};\\
 &  & \Omega_{k}:=\{\sigma\in\Omega_{1}^{k}:p_{\sigma_{1}\sigma_{2}}\cdots p_{\sigma_{k-1}\sigma_{k}}>0\},\; k\geq2;\\
 &  & \Omega^{*}:=\bigcup_{k\geq0}\Omega_{k},\;\Omega_{\infty}:=\{\sigma\in\Omega_{1}^{\mathbb{N}}:p_{\sigma_{h}\sigma_{h+1}}>0\;\;{\rm for\; all}\;\; h\geq1\}.
\end{eqnarray*}
We denote by $|\sigma|$ the length of $\sigma$, namely, $|\sigma|:=k$
for $\sigma\in\Omega_{k}$ and $|\theta|:=0$. For every word $\sigma=(\sigma_{1},\ldots,\sigma_{n})$
with $n\geq k$ or $\sigma\in\Omega_{\infty}$, we write $\sigma|_{k}:=(\sigma_{1},\ldots,\sigma_{k})$.
If $\sigma,\omega\in\Omega^{*}$ and $(\sigma_{|\sigma|},\omega_{1})\in\Omega_{2}$,
then we define 
\[
\sigma\ast\omega=(\sigma_{1},\sigma_{2},\ldots,\sigma_{|\sigma|},\omega_{1},\ldots,\omega_{|\omega|})\in\Omega_{|\sigma|+|\omega|}.
\]
Let $J_{i},1\leq N$ be non-empty compact subsets of $\mathbb{R}^{q}$
with $J_{i}=\overline{{\rm int}(J_{i})}$ for all $1\leq i\leq N$,
where $\overline{B}$ and ${\rm int}(B)$ respectively denote the
closure and interior in $\mathbb{R}^{q}$ of a set $B\subset\mathbb{R}^{q}$.
For convenience, we always assume that 
\[
{\rm diam}(J_{i})=1\;\;{\rm for\; all}\;\;1\leq i\leq N.
\]
Let $(c_{ij})_{N\times N}$ be a non-negative matrix such that $c_{ij}>0$
if and only if $p_{ij}>0$. For each pair $1\leq i,j\leq N$ with
$p_{ij}>0$, let $T_{ij}$ be a contracting similitude on $\mathbb{R}^{q}$
of contraction ratio $c_{ij}$. Assume that, $T_{ij}(J_{j}),$ $(i,j)\in\Omega_{2}$,
are non-overlapping subsets of $J_{i}$. By \citep[Corollary 3.5]{BED:89}
(see also \citep[Theorem 3]{MW:88}), there exists a unique vector
compact sets $(K_{i})_{i=1}^{N}\subset\prod_{i=1}^{N}J_{i}$ such
that 
\begin{equation}
K_{i}=\bigcup_{j:(i,j)\in\Omega_{2}}T_{ij}(K_{j}),\;\;1\leq i\leq N.\label{invariance1}
\end{equation}
We call $K:=\bigcup_{i=1}^{N}K_{i}$ the recurrent self-similar set
associated with the contracting similitudes $T_{ij},1\leq i,j\leq N$.
One can see that $K$ is also a \emph{map-specified MW-fractal} which
is defined in terms of a directed graph \citep{MW:88}.

Assume that $P$ is irreducible. Let $v=(q_{i})_{i=1}^{N}$ be the
unique normalized positive left eigenvector of $P$ with respect to
the Perron-Frobenius eigenvalue $1$, or equivalently, 
\[
\sum_{i=1}^{N}v_{i}=1,\; v_{i}>0,\;1\leq i\leq N;\;\;\sum_{i=1}^{N}v_{i}p_{ij}=v_{j}.
\]
We accordingly have a unique vector $(\nu_{i})_{i=1}^{N}$ of probability
measures such that 
\begin{equation}
\nu_{i}=\sum_{j:(i,j)\in\Omega_{2}}p_{ij}\nu_{j}\circ T_{ij}^{-1}\label{invariance2}
\end{equation}
and $\nu_{i}$ is supported by $K_{i}$ for each $1\leq i\leq N$.
Hence, we get a Markov-type measure $\nu:=\sum_{i=1}^{N}p_{i}\nu_{i}$
supported on $K$.

Assuming the irreducibility of the corresponding transition matrices
(or strong connectedness of the corresponding graphs), Lindsay has
studied the quantization problem for Markov-type measures on map-specified
graph directed fractals in \citep{LJL:01}; in there he expressed
the quantization dimension $D_{r}$ in terms of temperature functions
and showed that the $D_{r}$-dimensional upper quantization coefficient
is finite. Let us note the following facts:\renewcommand{\labelenumi}{\arabic{enumi}. }

\begin{enumerate} 

\item  the arguments in \citep{LJL:01} depend on the invariance
properties (\ref{invariance1}) and (\ref{invariance2}); these arguments
are not applicable to ratio-specified cases due to the absence of
the invariance properties;

\item the interesting cases, where the transition matrices are reducible,
have not been explored.

\end{enumerate}

In the present paper, we consider the Markov-type measures $\mu$
on a class of \emph{ratio-specified graph directed fractals} $E$.
Mauldin and Williams \citep[Theorem 4]{MW:88} have established a
necessary and sufficient condition for the Hausdorff measure of a
graph directed fractal to be positive and finite. Significantly inspired
by this result, we will establish a necessary and sufficient condition
for the upper and lower quantization coefficient of $\mu$ to be both
positive and finite, allowing the corresponding transition matrices
to be reducible.

Let $J_{i},P=(p_{ij})_{N\times N}$, be given as above. We call $J_{i},$
$1\leq i\leq N$, cylinder sets of order one. For each $1\leq i\leq N$,
let $J_{ij},\;(i,j)\in\Omega_{2}$, be non-overlapping subsets of
$J_{i}$ such that $J_{ij}$ is geometrically similar to $J_{j}$
with ${\rm diam}(J_{ij})/{\rm diam}(J_{j})=c_{ij}$. We call these
sets cylinder sets of order two. Assume that cylinder sets of order
$k$ are determined, namely, for each $\sigma:=\left(i_{1},\ldots,i_{k}\right)\in\Omega_{k}$,
we have a cylinder set $J_{\sigma}$. Let $J_{\sigma\ast i_{k+1}},$
$\sigma\ast i_{k+1}\in\Omega_{k+1}$, be non-overlapping subsets of
$J_{\sigma}$ such that $J_{\sigma\ast i_{k+1}}$ is geometrically
similar to $J_{\sigma}$ with ${\rm diam}(J_{\sigma\ast i_{k+1}})/{\rm diam}(J_{\sigma})=c_{i_{k}i_{k+1}}$.
Inductively, cylinder sets of order $k$ are determined for all $k\geq1$.
The \emph{ratio specified MW-fractal }is then given by 
\[
E:=\bigcap_{k\geq1}\bigcup_{\sigma\in\Omega_{k}}J_{\sigma}.
\]
Note that we only fix the contraction ratios $c_{ij},$ $1\leq i\leq j\leq N$,
and we do not fix the similarity mappings, so a ratio-specified MW-fractal
typically does not enjoy the invariance property (\ref{invariance1})
of $K$.

Let $(\chi_{i})_{i=1}^{N}$ be an arbitrary probability vector with
$\min_{1\leq i\leq N}\chi_{i}>0$. By Kolmogorov consistency theorem,
there exists a unique probability measure $\widetilde{\mu}$ on $\Omega_{\infty}$
such that $\widetilde{\mu}([\sigma])=\chi_{\sigma_{1}}p_{\sigma_{1}\sigma_{2}}\cdots p_{\sigma_{k-1}\sigma_{k}}$
for every $k\geq1$ and $\sigma=(\sigma_{1},\ldots,\sigma_{k})\in\Omega_{k}$,
where $[\sigma]:=\{\omega\in\Omega_{\infty}:\omega|_{|\sigma|}=\sigma\}$.
Let $\pi$ denote the projection from $\Omega_{\infty}$ to $E$:
$\pi(\sigma)=x$, where 
\[
\{x\}:=\bigcap_{k\geq1}J_{\sigma|_{k}},\;\;{\rm for}\;\;\sigma\in\Omega_{\infty}.
\]
To overcome the difficulty caused by the absence of invariance properties,
we assume the following separation property for $E$: there is some
constant $0<t<1$ such that for every $\sigma\in\Omega^{*}$ and distinct
$i_{1},i_{2}\in\Omega_{1}$ with $p_{\sigma_{|\sigma|}i_{l}}>0,$$l=1,2$,
\begin{equation}
d(J_{\sigma\ast i_{1}},J_{\sigma\ast i_{2}}):=\inf\left\{ \left|x-y\right|:x\in J_{\sigma\ast i_{1}},y\in J_{\sigma\ast i_{2}}\right\} \geq t\max\{|J_{\sigma\ast i_{1}}|,|J_{\sigma\ast i_{2}}|\}.\label{g4}
\end{equation}
Here $|A|$ denotes the diameter of a set $A\subset\mathbb{R}^{q}$.
Under this assumption, $\pi$ is a bijection. We consider the image
measure of $\widetilde{\mu}$ under the projection $\pi$: $\mu:=\widetilde{\mu}\circ\pi^{-1}$.
We call $\mu$ a Markov-type measure which satisfies 
\begin{eqnarray}
\mu(J_{\sigma})=\chi_{\sigma_{1}}p_{\sigma_{1}\sigma_{2}}\cdots p_{\sigma_{k-1}\sigma_{k}}\;\;{\rm for}\;\;\sigma=(\sigma_{1},\ldots,\sigma_{k})\in\Omega_{k}.\label{markovmeasure}
\end{eqnarray}
As there are infinitely many similitudes corresponding to given contraction
ratios $c_{ij}$, $\mu$ generally does not enjoy the invariance property
(\ref{invariance2}).

\subsection{Statement of main results}

Before we state our main result, we need to recall some facts on spectral
radius of matrices and some notations on directed graphs.

For $1\leq i,j\leq N$, we define $a_{ij}(s):=(p_{ij}c_{ij}^{r})^{s}$.
Then we get an $N\times N$ matrix $A(s)=(a_{ij}(s))_{N\times N}$.
Let $\psi(s)$ denote the spectral radius of $A(s)$. By \citep[Theorem 2]{MW:88},
$\psi(s)$ is continuous and strictly decreasing. Note that, by the
assumption (\ref{cardpij>0}), $\psi(0)\geq2$; by Perron-Frobenius
theorem, we have, 
\begin{eqnarray*}
\psi(1)\leq\max_{1\leq i\leq N}\sum_{j=1}^{N}a_{ij}(1)<\max_{1\leq i\leq N}\sum_{j=1}^{N}p_{ij}=1.
\end{eqnarray*}
Intermediate-value theorem implies that there exists a unique number
$\xi\in(0,1)$ such that $\psi(\xi)=1$. Thus, for every $r>0$, there
exists a unique positive number $s_{r}$ such that $\psi\left(s_{r}/(s_{r}+r)\right)=1$.

As in \citep{MW:88}, we consider the directed graph $G$ associated
with the transition matrix $(p_{ij})_{N\times N}$. Namely, $G$ has
vertices $1,2,\ldots,N$; there is an edge from $i$ to $j$ if and
only if $p_{ij}>0$. In the following, we will simply denote by $G=\{1,\ldots,N\}$
both the directed graph and its vertex sets. We also write 
\[
b_{ij}(s):=(p_{ij}c_{ij}^{r})^{s/(s+r)},\;\; A_{G,s}:=(b_{ij}(s))_{N\times N};\;\;\Psi_{G}(s):=\psi\left(s/(s+r)\right).
\]
We also refer to an element $\left(i_{1},\ldots,i_{k}\right)\in\Omega_{k}$
as a path in $G$. We call $H\subset G$, with edges inherited from
$G$, a subgraph of $G$. A subgraph $H$ of $G$ is called strongly
connected if for very pair $i_{1},i_{2}\in H$, there exists a path
$\gamma$ in $H$ which begins at $i_{1}$ and ends at $i_{2}$. A
strongly connected component of $G$ refers to a maximal strongly
connected subgraph. Let ${\rm SC}(G)$ denote the set of all strongly
connected components of $G$. For $H_{1},H_{2}\in\SC(G)$, we write
$H_{1}\prec H_{2}$, if there is a path $\gamma=\left(i_{1},\ldots,i_{k}\right)$
in $G$ such that $i_{1}\in H_{1}$ and $i_{k}\in H_{2}$. If we have
neither $H_{1}\prec H_{2}$ nor $H_{2}\prec H_{1}$, then we say $H_{1},H_{2}$
are incomparable.

For every $H\in\SC(G)$, we denote by $A_{H,s}$ the sub-matrix $(b_{ij}(s))_{i,j\in H}$
of $A_{G}(s)$. Let $\Psi_{H}(s)$ be the spectral radius of $A_{H,s}$
and $s_{r}(H)$ be the unique positive number satisfying $\Psi_{H}(s_{r}(H))=1$.
With assumption (\ref{cardpij>0}), one can see by pigeon-hole principle
that $G$ has at least one strongly connected component $H$ with
${\rm card}(H)\geq2$. As our main result, we will prove 
\begin{thm}
\label{mthm1} Assume that (\ref{cardpij>0}) and (\ref{g4}) are
satisfied and let $\mu$ be the Markov-type measure defined in (\ref{markovmeasure}),
and let $s_{r}$ be the unique positive number satisfying $\Psi_{G}(s_{r})=1$.
Then we have, 
\[
D_{r}(\mu)=s_{r}~\mbox{and }~\underline{Q}_{r}^{s_{r}}(\mu)>0.
\]
 Furthermore, $\overline{Q}_{r}^{s_{r}}(\mu)<\infty$ if and only
if $\mathcal{M}:=\{H\in\SC(G):s_{r}(H)=s_{r}\}$ consists only of
incomparable elements; otherwise, we have $\underline{Q}_{r}^{s_{r}}(\mu)=\infty$. 
\end{thm}
At this point, we remark that, although the quantization problem for
probability measures and the Hausdorff measure of sets are two substantially
different objects, we benefit significantly from some methods previously
developed in \citep{MW:88}.

\section{Notations and preliminary facts}

For every $k\geq2$ and $\sigma=(\sigma_{1},\ldots,\sigma_{k})\in\Omega_{k}$,
we write 
\begin{eqnarray*}
\sigma^{-}:=(\sigma_{1},\ldots,\sigma_{k-1});\;\; p_{\sigma}:=p_{\sigma_{1}\sigma_{2}}\cdots p_{\sigma_{k-1}\sigma_{k}},\; c_{\sigma}:=c_{\sigma_{1}\sigma_{2}}\cdots c_{\sigma_{k-1}\sigma_{k}}.
\end{eqnarray*}
If $|\sigma|=1$, we set $\sigma^{-}=\theta$, where $\theta$ denotes
the empty word; we also define $p_{\sigma}:=1,c_{\sigma}:=1$ for
$\sigma\in\Omega_{1}\cup\left\{ \theta\right\} $. If $\sigma,\omega\in\Omega^{*}$
satisfy $|\sigma|\leq|\omega|$ and $\sigma=\omega|_{|\sigma|}$,
then we call $\omega$ a descendant of $\sigma$ and write $\sigma\prec\omega$.
Two words $\sigma,\omega\in\Omega^{*}$ are said to be incomparable
if neither $\sigma\prec\omega$, nor $\omega\prec\sigma$. A finite
subset $\Gamma$ of $\Omega^{*}$ is called a finite antichain if
any two words in $\Gamma$ are incomparable; a finite antichain $\Gamma$
is said to be maximal, if every $\tau\in\Omega_{\infty}$ is the descendant
of some word $\sigma\in\Gamma$. Set 
\begin{eqnarray*}
\underline{p}:=\min_{(i,j)\in\Omega_{2}}p_{ij},\;\underline{c}:=\min_{(i,j)\in\Omega_{2}}c_{ij},\;\overline{p}:=\max_{(i,j)\in\Omega_{2}}p_{ij},\;\overline{c}:=\max_{(i,j)\in\Omega_{2}}c_{ij}.
\end{eqnarray*}
For $r>0$, let $\eta:=\underline{p}\underline{c}^{r}$. To study
the quantization error $e_{n,r}(\mu)$, we define 
\begin{eqnarray}
\Lambda_{j,r}:=\{\sigma\in\Omega^{*}:p_{\sigma^{-}}c_{\sigma^{-}}^{r}\geq\eta^{j}>p_{\sigma}c_{\sigma}^{r}\}.\label{lambdajr}
\end{eqnarray}
Then $(\Lambda_{j,r})_{j=1}^{\infty}$ is a sequence of finite maximal
antichains. This type of sets were constructed by Graf and Luschgy
in their work on the quantization for self-similar measures (cf. \citep{GL:00}).
The spirit of these constructions is to seek some kind of uniformity
while general measures are not uniform. We define 
\begin{eqnarray*}
 &  & \phi_{j,r}:={\rm card}(\Lambda_{j,r}),\; l_{1j}:=\min_{\sigma\in\Lambda_{j,r}}|\sigma|,\; l_{2j}:=\max_{\sigma\in\Lambda_{j,r}}|\sigma|;\\
 &  & \underline{P}_{r}^{s}(\mu):=\liminf_{j\to\infty}\phi_{j,r}^{\frac{r}{s}}e_{\phi_{j,r},r}^{r}(\mu),\;\;\overline{P}_{r}^{s}(\mu):=\limsup_{j\to\infty}\phi_{j,r}^{\frac{r}{s}}e_{\phi_{j,r},r}^{r}(\mu).
\end{eqnarray*}

\begin{lem}
\label{g9} We have 
\[
\underline{Q}_{r}^{s}(\mu)>0\iff\underline{P}_{r}^{s}(\mu)>0~~~~\mbox{and }~~~~\overline{Q}_{r}^{s}(\mu)<\infty\iff\overline{P}_{r}^{s}(\mu)<\infty.
\]
 \end{lem}
\begin{proof}
Let $N_{1}:=\min\{h\in\mathbb{N}:(\overline{p}\overline{c}^{r})^{h}<\eta\}$.
For every $\sigma\in\Lambda_{j,r}$, we have, $p_{\sigma}c_{\sigma}^{r}<\eta^{j}$.
Hence, for every $\omega\in\Omega_{N_{1}}$ with $(\sigma_{|\sigma|},\omega_{1})\in\Omega_{2}$,
we have 
\begin{eqnarray*}
p_{\sigma\ast\omega}c_{\sigma\ast\omega}^{r}=(p_{\sigma}c_{\sigma}^{r})(p_{\sigma_{|\sigma|}\omega_{1}})p_{\omega}c_{\omega}^{r}\leq(p_{\sigma}c_{\sigma}^{r})(\overline{p}\overline{c}^{r})^{N_{1}}<\eta^{j+1}.
\end{eqnarray*}
Hence, $\Lambda_{j+1,r}\subset\bigcup_{h=1}^{N_{1}}\bigcup_{\sigma\in\Lambda_{j,r}}\Gamma(\sigma,h)$,
where 
\begin{eqnarray}
\Gamma(\sigma,h):=\left\{ \omega\in\Omega_{|\sigma|+h}:\sigma\prec\omega\right\} .\label{g30}
\end{eqnarray}
It follows that $\phi_{j,r}\leq\phi_{j+1,r}\leq N^{N_{1}}\phi_{j,r}$.
This and \citep[Lemma 2.4]{Zhu:12} completes the proof of the lemma. 
\end{proof}
If the infimum in (\ref{quanerrordef}) is attained at some $\alpha$
with ${\rm card}(\alpha)\leq n$, we call $\alpha$ an $n$-optimal
set for $\nu$ of order $r$. The collection of all $n$-optimal sets
for $\nu$ of order $r$ is denoted by $C_{n,r}(\nu)$. For two sequences
$(a_{n})_{n=1}^{\infty}$ and $(b_{n})_{n=1}^{\infty}$ of positive
real numbers, we write $a_{n}\ll b_{n}$ if there is some constant
$B$ independent of $n$ such that $a_{n}\leq B\cdot b_{n}$. If $a_{n}\ll b_{n}$
and $b_{n}\ll a_{n}$ we write $a_{n}\asymp b_{n}$. 

For every $k\geq1$ and a vector $w=(w_{i})_{i=1}^{k}\in\mathbb{R}^{k}$,
we define 
\[
\overline{w}:=\max_{1\leq i\leq k}w_{i},\;\underline{w}:=\min_{1\leq i\leq k}w_{i}.
\]

\begin{lem}
For all large $j\geq1$, we have 
\begin{eqnarray}
e_{\phi_{j,r},r}^{r}(\mu)\asymp\sum_{\sigma\in\Lambda_{j,r}}p_{\sigma}c_{\sigma}^{r}.\label{characterization}
\end{eqnarray}
\end{lem}
\begin{proof}
For every $\sigma\in\Lambda_{j,r}$, let $a_{\sigma}$ be an arbitrary
point in $J_{\sigma}$. We have 
\begin{eqnarray*}
e_{\phi_{j,r},r}^{r}(\mu) & \leq & \sum_{\sigma\in\Lambda_{j,r}}\int_{J_{\sigma}}d(x,a_{\sigma})^{r}d\mu(x)\\
 & \leq & \sum_{\sigma\in\Lambda_{j,r}}\mu(J_{\sigma})|J_{\sigma}|^{r}=\sum_{\sigma\in\Lambda_{j,r}}\chi_{\sigma_{1}}p_{\sigma}c_{\sigma}^{r}\leq\overline{\chi}\sum_{\sigma\in\Lambda_{j,r}}p_{\sigma}c_{\sigma}^{r}.
\end{eqnarray*}
Using (\ref{g4}) and the method in \citep[Lemma 3]{Zhu:08}, one
can find a constant $L\geq1$, which is independent of $j$, and a
set $\beta(\sigma)$ with ${\rm card}(\beta(\sigma))\leq L$ such
that 
\begin{eqnarray}
e_{\phi_{j,r},r}^{r}(\mu)\geq\sum_{\sigma\in\Lambda_{j,r}}\int_{J_{\sigma}}d(x,\beta(\sigma))^{r}d\mu(x).\label{s30}
\end{eqnarray}
Then by (\ref{cardpij>0}) and the arguments in \citep[Lemma 4]{Zhu:08},
one may find a constant $D>0$ which is independent of $\sigma\in\Omega^{*}$,
such that 
\begin{eqnarray}
\int_{J_{\sigma}}d(x,\beta(\sigma))^{r}d\mu(x)\geq D\mu(J_{\sigma})|J_{\sigma}|^{r}\geq D\underline{\chi}p_{\sigma}c_{\sigma}^{r}.\label{s31}
\end{eqnarray}
By (\ref{s30}) and (\ref{s31}), we conclude that $e_{\phi_{j,r},r}^{r}(\mu)\geq D\underline{\chi}\sum_{\sigma\in\Lambda_{j,r}}p_{\sigma}c_{\sigma}^{r}$.
\end{proof}

\section{Proof of Theorem \ref{mthm1}}

We will treat the irreducible and non-irreducible case separately.

\subsection{Markov measures with irreducible transition matrix}
\begin{lem}
\label{lem2} Assume that $P=(p_{ij})_{N\times N}$ is irreducible.
Then there exist constants $\delta_{i}>0,i=1,2$ such that, for every
finite maximal antichain $\Gamma\subset\Omega^{*}$, 
\[
\delta_{1}\leq\sum_{\sigma\in\Gamma}\left(p_{\sigma}c_{\sigma}^{r}\right)^{s_{r}/\left(s_{r}+r\right)}\leq\delta_{2}.
\]
\end{lem}
\begin{proof}
As $\Psi_{G}(s_{r})=1$ and $A_{G,s_{r}}$ is a non-negative irreducible
matrix, by Perron-Frobenius theorem, $1$ is an eigenvalue of $A_{G,s_{r}}$
and there is a positive vector $\xi=(\xi_{i})_{i=1}^{N}$ with $\sum_{i=1}^{N}\xi_{i}=1$
such that $A_{G,s_{r}}\xi=\xi$. As before, for $k\geq2$ and $\sigma\in\Omega_{k}$,
we set $[\sigma]:=\{\omega\in\Omega_{\infty}:\omega|_{|\sigma|}=\sigma\}$.
We define 
\[
\nu_{1}([\sigma]):=\left(p_{\sigma}c_{\sigma}^{r}\right)^{s_{r}/\left(s_{r}+r\right)}\xi_{\sigma_{|\sigma|}}.
\]
Then one can easily see 
\begin{eqnarray*}
\sum_{i=1}^{N}\nu_{1}([\sigma\ast i]) & = & (p_{\sigma}c_{\sigma}^{r})^{s_{r}/\left(s_{r}+r\right)}\sum_{i=1}^{N}(p_{\sigma_{|\sigma|}i}c_{\sigma_{|\sigma|}i}^{r})^{s_{r}/\left(s_{r}+r\right)}\xi_{i}=\nu_{1}([\sigma]),\\
\sum_{\sigma\in\Omega_{2}}(p_{\sigma}c_{\sigma}^{r})^{s_{r}/\left(s_{r}+r\right)}\xi_{\sigma_{|\sigma|}} & = & \sum_{i=1}^{N}\sum_{j=1}^{N}(p_{ij}c_{ij}^{r})^{s_{r}/\left(s_{r}+r\right)}\xi_{j}=\sum_{i=1}^{N}\xi_{i}=1.
\end{eqnarray*}
Thus, by Kolmogorov consistency theorem, $\nu_{1}$ extends a probability
measure on $\Omega_{\infty}$. Since $\Gamma$ is a finite maximal
antichain, we have 
\[
\sum_{\sigma\in\Gamma}\left(p_{\sigma}c_{\sigma}^{r}\right)^{s_{r}/\left(s_{r}+r\right)}\xi_{\sigma_{|\sigma|}}=\nu_{1}(\Omega_{\infty})=1.
\]
As an immediate consequence, we have 
\[
\overline{\xi}^{-1}\leq\sum_{\sigma\in\Gamma}\left(p_{\sigma}c_{\sigma}^{r}\right)^{s_{r}/\left(s_{r}+r\right)}\leq\underline{\xi}^{-1}.
\]
The lemma follows by setting $\delta_{1}:=\overline{\xi}^{-1}$ and
$\delta_{2}:=\underline{\xi}^{-1}$.\end{proof}
\begin{prop}
\label{mthm2'} Assume that $P=(p_{ij})_{N\times N}$ is irreducible.
Then we have 
\begin{equation}
0<\underline{Q}_{r}^{s_{r}}(\mu)\leq\overline{Q}_{r}^{s_{r}}(\mu)<\infty.\label{s18}
\end{equation}
\end{prop}
\begin{proof}
By Lemma \ref{lem2}, for $j\geq1$, we see 
\begin{eqnarray*}
\phi_{j,r}\eta^{js_{r}/(s_{r}+r)}\asymp\sum_{\sigma\in\Lambda_{j,r}}\left(p_{\sigma}c_{\sigma}^{r}\right)^{s_{r}/(s_{r}+r)}\asymp1,\;\;\mbox{implying }\;\eta^{j}\asymp\phi_{j,r}^{-(s_{r}+r)/s_{r}}.
\end{eqnarray*}
This, together with (\ref{characterization}), leads to 
\begin{eqnarray*}
e_{\phi_{j,r}}^{r}(\mu)\asymp\sum_{\sigma\in\Lambda_{j,r}}p_{\sigma}c_{\sigma}^{r}\asymp\phi_{j,r}\cdot\eta^{j}\asymp\phi_{j,r}\cdot\phi_{j,r}^{-(s_{r}+r)/s_{r}}=\phi_{j,r}^{-r/s_{r}}.
\end{eqnarray*}
It follows that $\phi_{j,r}^{r/s_{r}}e_{\phi_{j,r}}^{r}(\mu)\asymp1$.
Hence, (\ref{s18}) follows by Lemma \ref{g9}.
\end{proof}

\subsection{Markov measures with reducible transition matrix}

For every $H\in\SC(G)$, we write 
\begin{eqnarray*}
H_{\infty} & := & \left\{ \sigma\in\Omega_{\infty}:\sigma_{i}\in H\;{\rm for}\; i\geq1\right\} ,\;\; H^{*}:=\bigcup_{k=1}^{\infty}H^{k};\\
H_{k}(i) & := & \left\{ \sigma\in H^{k}:\sigma_{1}=i\right\} ,\;\; H^{*}(i):=\bigcup_{k=1}^{\infty}H_{k}(i);\\
H_{\infty}(i) & := & \left\{ \sigma\in H_{\infty}:\sigma_{1}=i\right\} ;\;\; i\in H.
\end{eqnarray*}
We define finite maximal antichains in $H^{*}$ or $H^{*}(i)$ in
the same way as we did for those in $\Omega^{*}$.
\begin{lem}
[cf. \citep{MW:88}] \label{factor} We have $s_{r}=\max_{H\in\SC(G)}s_{r}(H)$. \end{lem}
\begin{proof}
By the factor theorem \citep[pp. 813]{MW:88}, we have 
\begin{equation}
0={\rm det}(I-A_{G,s_{r}})=\prod_{H\in\SC(G)}{\rm det}(I-A_{H,s_{r}}).\label{s12}
\end{equation}
Since $A_{G,s_{r}}$ is non-negative, by the definition of $s_{r}$,
$1$ is an eigenvalue of $A_{G,s_{r}}$. For every $H\in\SC(G)$,
we know that $A_{H,s_{r}(H)}$ is irreducible since $H$ is strongly
connected. Thus, there exists an eigenvalue $\lambda_{H}>0$ which
equals the spectral radius $\Psi_{H}(s_{r})$ of $A_{H,s_{r}}$, namely,
$\lambda_{H}$ is the Perron root of $A_{H,s_{r}}$. It follows that
$1=\max_{H\in{\rm SC}(G)}\lambda_{H}$ and there exists an $H\in\SC(G)$
such that $\lambda_{H}=1$. If $\lambda_{H}=1$, then $s_{r}=s_{r}(H)$;
otherwise, we have $\lambda_{H}=\Psi_{H}(s_{r})<1$. By \citep[Theorem 2]{MW:88},
$\Psi_{H}(s)$ is strictly decreasing with respect to $s$. Thus,
$s_{r}>s_{r}(H)$. Combining the above analysis, the lemma follows. \end{proof}
\begin{rem*}
The factor theorem is an easy consequence of the following facts.
For the non-negative reducible matrix $A_{G,s}$, there exists some
permutation matrix $T$ (which is necessarily orthogonal) such that
$TA_{G,s}T^{-1}$ is a block upper triangular matrix, where the blocks
on the diagonal are either irreducible matrices or $1\times1$ null
matrices; thus the set of the eigenvalues of $A_{G,s}$ are exactly
the union of those of all the blocks on the diagonal. 

Further, the non-zero blocks on the diagonal are the images of the
maximal irreducible sub-matrices of $A_{G,s}$ corresponding to the
strongly connected components under symmetric permutations which are
orthogonal and preserve eigenvalues. Therefore, if we denote by $m$
(possibly zero) the number of $1\times1$ null matrices on the diagonal
of $TA_{G,s}T^{-1}$, then 
\[
{\rm det}(\lambda I-A_{G,s})=\lambda^{m}\prod_{H\in\SC(G)}{\rm det}(\lambda I-A_{H,s}).
\]
\end{rem*}
\begin{lem}
\label{g14} There exists constant $M_{1},M_{2}>0$ such that 
\begin{equation}
M_{1}\leq\sum_{\sigma\in\Gamma}\left(p_{\sigma}c_{\sigma}^{r}\right)^{s_{r}\left(H\right)/(s_{r}\left(H\right)+r)}\leq M_{2}.\label{g13}
\end{equation}
 for every finite maximal antichain in $H^{*}$ or $H^{*}(i),i\in H$,
with $H\in{\rm SC}(G)$.\end{lem}
\begin{proof}
By Lemma \ref{lem2}, for each $H$, one can choose a constant $M_{H}$
such that the second inequality in (\ref{g13}) holds with $M_{H}$
in place of $M_{2}$. Set $M_{2}:=\max\{M_{H}:H\in\SC(G)\}$. Then
for every $H\in\SC(G)$ and every finite maximal antichain $\Gamma$
in $H^{*}$, we have 
\begin{eqnarray}
\sum_{\sigma\in\Gamma}\left(p_{\sigma}c_{\sigma}^{r}\right)^{s_{r}\left(H\right)/(s_{r}\left(H\right)+r)}\leq M_{2}\label{g18}
\end{eqnarray}
Since every finite maximal antichain $\Gamma(i)$ in $H^{*}(i)$ is
contained in a finite maximal antichain in $H^{*}$, we conclude that
the second inequality of (\ref{g13}) also holds for such a $\Gamma(i)$.
Next we need to choose $M_{1}$ such that the first inequality also
holds for such a $\Gamma(i)$.

Let $H\in\SC(G)$ and $i\in H$. We denote by $\zeta$ the unique
normalized positive right eigenvector of $A_{H,s_{r}(H)}$ with respect
to the Perron-Frobenius eigenvalue $1$. We consider the measure $\nu_{1,i}$
on $H_{\infty}(i)$ satisfying 
\begin{eqnarray*}
\nu_{1,i}([\sigma]) & := & \left(p_{\sigma}c_{\sigma}^{r}\right)^{s_{r}\left(H\right)/(s_{r}\left(H\right)+r)}\zeta_{\sigma_{|\sigma|}},\;\sigma\in H^{*}(i);\\
\nu_{1,i}(H_{\infty}(i) & = & \sum_{j=1}^{N}\left(p_{ij}c_{ij}^{r}\right)^{s_{r}\left(H\right)/(s_{r}\left(H\right)+r)}\zeta_{j}=\zeta_{i},
\end{eqnarray*}
where $[\sigma]:=\{\omega\in H_{\infty}(i):\omega|_{|\sigma|}=\sigma\}$.
Then for every finite maximal antichain $\Gamma(i)\subset H^{*}(i)$,
we have 
\[
\overline{\zeta}^{-1}\zeta_{i}\leq\sum_{\sigma\in\Gamma(i)}\left(p_{\sigma}c_{\sigma}^{r}\right)^{s_{r}\left(H\right)/(s_{r}\left(H\right)+r)}\leq\underline{\zeta}^{-1}\zeta_{i}.
\]
Let $m_{H}:=\min\{\overline{\zeta}^{-1}\zeta_{i}:i\in H\}$ and $M_{1}:=\min\{m_{H}:H\in\SC(G)\}$.
Then (\ref{g13}) holds for every every finite maximal antichain $\Gamma(i)$
in $H^{*}(i)$ and finite maximal antichain in $H^{*}$. This completes
the proof of the lemma.\end{proof}
\begin{lem}
\label{g7} Let $H\in\SC(G)$ and $i\in H$. Then $\underline{Q}_{r}^{s_{r}(H)}(\mu(\cdot|J_{i}))>0$.\end{lem}
\begin{proof}
For each $k\geq1$, we write 
\begin{eqnarray}
 &  & \Lambda_{k,r}(i):=\left\{ \sigma\in\Omega^{*}:\sigma_{1}=i,p_{\sigma^{-}}c_{\sigma^{-}}^{r}\geq\eta^{k}>p_{\sigma}c_{\sigma}^{r}\right\} ,\label{s20}\\
 &  & \Lambda_{k,r}(i,H^{c}):=\left\{ \sigma\in\Lambda_{k,r}(i):\sigma_{h}\notin H\;\mbox{for some }\; h\right\} .\nonumber 
\end{eqnarray}
For $\sigma\in\Lambda_{k,r}(i,H^{c})$, we write $h(\sigma):=\min\{h:\sigma_{h}\notin H\}$.
Then, since $H$ is a strongly connected component, we deduce that
$\sigma_{l}\notin H$ for all $l\geq h(\sigma)$. Note that $A_{H,s_{r}(H)}$
is irreducible and that $\Lambda_{k,r}(i)\setminus\Lambda_{k,r}(i,H^{c})$
is a maximal finite antichain in $H^{*}(i)$. Hence, by Lemma \ref{g14},
we have 
\[
\sum_{\sigma\in\Lambda_{k,r}(i)\setminus\Lambda_{k,r}(i,H^{c})}\left(p_{\sigma}c_{\sigma}^{r}\right)^{s_{r}\left(H\right)/(s_{r}\left(H\right)+r)}\geq M_{1}>0.
\]
Let $\phi_{k,r}(i)$ denote the cardinality of $\Lambda_{k,r}(i)$.
As we did for (\ref{characterization}), one can show 
\[
e_{\phi_{k,r}(i),r}^{r}\big(\mu(\cdot|J_{i})\big)\gg\sum_{\sigma\in\Lambda_{k,r}(i)}p_{\sigma}c_{\sigma}^{r}.
\]
By H\"older's inequality for exponent less than one, we have 
\begin{eqnarray*}
e_{\phi_{k,r}(i),r}^{r}\big(\mu(\cdot|J_{i})\big) & \gg & \bigg(\sum_{\sigma\in\Lambda_{k,r}(i)}(p_{\sigma}c_{\sigma}^{r})^{\frac{s_{r}(H)}{s_{r}(H)+r}}\bigg)^{\frac{s_{r}(H)+r}{s_{r}(H)}}\cdot\phi_{k,r}(i)^{-r/s_{r}\left(H\right)}\\
 & \geq & \bigg(\sum_{\sigma\in\Lambda_{k,r}(i)\setminus\Lambda_{k,r}(i,H^{c})}(p_{\sigma}c_{\sigma}^{r})^{\frac{s_{r}(H)}{s_{r}(H)+r}}\bigg)^{\frac{s_{r}(H)+r}{s_{r}(H)}}\cdot\phi_{k,r}(i)^{-r/s_{r}\left(H\right)}\\
 & \gg & \phi_{k,r}(i)^{-r/s_{r}\left(H\right)}.
\end{eqnarray*}
This and Lemma \ref{g9} yields that $\underline{Q}_{r}^{s_{r}(H)}(\mu(\cdot|J_{i}))>0$.
The lemma follows. \end{proof}
\begin{prop}
\label{mthm2} For any Markov-type measure $\mu$ as defined in (\ref{markovmeasure})
we have $D_{r}(\mu)=s_{r}$ and $\underline{Q}_{r}^{s_{r}}(\mu)>0$.\end{prop}
\begin{proof}
Let $s>s_{r}$. By \citep[Theorem 2]{MW:88}, we have, $\Psi_{G}(s)<\Psi_{G}(s_{r})=1$.
Let $u=(u_{i})_{i=1}^{N}$ be the column vector with $u_{i}=1$ for
all $1\leq i\leq N$. We choose $t$ such that $\Psi_{G}(s)<t<1$.
By Gelfand's formula, we have 
\[
\lim_{k\to\infty}\|A_{G,s}^{k}u\|_{1}^{1/k}=\Psi_{G}(s)<t<1.
\]
Thus, for large $k$, we have that $\|A_{G,s}^{k}u\|_{1}<t^{k}$.
It follows that 
\begin{eqnarray}
\sum_{\sigma\in\Omega_{k}}\left(p_{\sigma}c_{\sigma}^{r}\right)^{s/(s+r)}=\|A_{G,s}^{k-1}u\|_{1}<t^{k-1}.\label{s10}
\end{eqnarray}
Let $\Lambda_{j,r}$ be as defined in (\ref{lambdajr}). It is immediate
to see that, there exist two constants $A_{1},A_{2}>0$ such that
\[
A_{1}j\leq l_{1j}\leq l_{2j}<A_{2}j.
\]
Applying (\ref{s10}) to every $l_{1j}\leq k\leq l_{2j}$, we deduce
\begin{eqnarray*}
\sum_{\sigma\in\Lambda_{j,r}}\left(p_{\sigma}c_{\sigma}^{r}\right)^{s/(s+r)} & \leq & \sum_{k=l_{1j}}^{l_{2j}}\sum_{\sigma\in\Omega_{k}}\left(p_{\sigma}c_{\sigma}^{r}\right)^{s/(s+r)}\\
 & \leq & \sum_{k=l_{1j}}^{l_{2j}}t^{k-1}\leq\frac{t^{l_{1j}-1}}{1-t}<1\;\;{\rm for\; large\;\; j}.
\end{eqnarray*}
Thus, by (\ref{lambdajr}), for all large $j$, we have $\phi_{j,r}\leq\eta^{-s\left(j+1\right)/\left(s+r\right)}$.
Also, by (\ref{characterization}), 
\begin{eqnarray*}
e_{\phi_{j,r},r}(\mu)\leq\overline{\chi}\sum_{\sigma\in\Lambda_{j,r}}(p_{\sigma}c_{\sigma}^{r})^{s/(s+r)}(p_{\sigma}c_{\sigma}^{r})^{r/(s+r)}\leq\overline{\chi}\eta^{rj/(s+r)}\leq\overline{\chi}\phi_{j,r}^{-r/s}.
\end{eqnarray*}
Thus, by Lemma \ref{g9}, we have, $\overline{D}_{r}(\mu)\leq s$.
Since $s>s_{r}$ was chosen arbitrarily, we obtain that $\overline{D}_{r}(\mu)\leq s_{r}$.

Let $H\in\mathcal{M}$. Then we have $s_{r}(H)=s_{r}$. We take an
arbitrary vertex $i_{0}\in H$ and consider the conditional probability
measure $\mu_{i_{0}}:=\mu(\cdot|J_{i_{0}})$. By Lemma \ref{g7},
we have, $\underline{Q}_{r}^{s_{r}}(\mu_{i_{0}})>0$. Hence, 
\begin{eqnarray*}
\underline{Q}_{r}^{s_{r}}(\mu)\geq\mu(J_{i_{0}})\underline{Q}_{r}^{s_{r}}(\mu_{i_{0}})\geq\underline{\chi}\,\underline{Q}_{r}^{s_{r}}(\mu_{i_{0}})>0.
\end{eqnarray*}
In particular, by \citep[Proposition 11.3]{GL:00}, we have, $\underline{D}_{r}(\mu)\geq s_{r}$.
Combing this and the first part of the proof, we conclude that $D_{r}(\mu)$
exists and equals $s_{r}$. This completes the proof of the theorem.
\end{proof}
Define the set $F:=G\setminus\bigcup_{H\in\mathcal{M}}H$ which is
possibly empty. Whenever $F\neq\emptyset$, there corresponds a sub-matrix
$A_{F,s_{r}}$ of $A_{G,s_{r}}$. We write 
\begin{eqnarray*}
F_{0}:=\{\theta\},\;\; F_{k}:=\{\sigma\in\Omega_{k}:\sigma_{h}\in F,1\leq h\leq k\},\; k\geq1;\;\; F^{*}:=\bigcup_{k=0}^{\infty}F_{k}.
\end{eqnarray*}

\begin{lem}
\label{s26} There exists a constant $t\in(0,1)$ such that for $n\in\mathbb{N}$
large 
\[
\sum_{\sigma\in F_{n}}\left(p_{\sigma}c_{\sigma}^{r}\right)^{s_{r}/\left(s_{r}+r\right)}\ll t^{n}.
\]
\end{lem}
\begin{proof}
Let $s_{r}(F)$ denote the unique number with $\Psi_{F}(s_{r}(F))$=1.
Then by Lemma \ref{factor} and the definition of $F$, we deduce
\[
s_{r}(F)=\max\{s_{r}(H):\SC(G)\ni H\subset F\}<s_{r}.
\]
According to \citep[Theorem 3]{MW:88}, $\Psi_{F}(s)$ is strictly
decreasing with respect to $s$. Thus, $\Psi_{F}(s_{r})<1$ and we
may choose some $t>0$ such that $\Psi_{F}(s_{r})<t<1$. Following
the proof of Proposition \ref{mthm2}, one can see that, there exists
a constant $k_{0}\in\mathbb{N}$ such that for all $k\geq k_{0}$
we have 
\begin{eqnarray*}
\sum_{\sigma\in F_{k}}(p_{\sigma}c_{\sigma}^{r})^{s_{r}/\left(s_{r}+r\right)}\leq t^{k-1},\;\; k\geq k_{1}.
\end{eqnarray*}
From this the lemma follows. 
\end{proof}

\begin{prop}
\label{mthm3} If $\mathcal{M}$ consists of pairwise incomparable
elements then $\overline{Q}_{r}^{s_{r}}(\mu)<\infty$.\end{prop}
\begin{proof}
First note that in this situation we have that any $\omega\in\Omega^{*}$
has the form $\omega=\nu'*\tau*\nu''$ where $\nu',\nu''\in F^{*}$
and $\tau\in H^{*}$ for some $H\in\mathcal{M}$. Further, for $M_{3}:=\min\left\{ k\in\mathbb{N}:\overline{p}\,\overline{c}^{r}/\underline{p}\underline{c}^{r}<k\right\} $,
$H\in\mathcal{M}$ and any choice $\nu',\nu''\in F^{*}$, we have
that 
\[
H_{j,r}^{*}\left(\nu',\nu''\right):=\left\{ \tau\in H^{*}:\nu'*\tau*\nu''\in\Lambda_{j,r}\right\} \subset\bigcup_{k=1}^{M_{3}}\Gamma_{k}^{H,j}\left(\nu',\nu''\right),
\]
where $\Gamma_{k}^{H,j}\left(\nu',\nu''\right)$ is some antichain
for each $k=1,\ldots,M_{3}$. With this notation and using Lemmata
\ref{g14} and \ref{s26} and the definition of $t$ and $M_{2}$
therein we estimate
\begin{align*}
 & \negmedspace\negmedspace\negmedspace\negmedspace\negmedspace\sum_{\omega\in\Lambda_{j,r}}\left(p_{\omega}c_{\omega}^{r}\right)^{s_{r}/\left(s_{r}+r\right)}\\
= & ~\sum_{H\in\mathcal{M}}\sum_{\nu',\nu''\in F^{*}}\sum_{\tau\in H_{j,r}^{*}\left(\nu',\nu''\right)}\left(p_{\nu'*\tau_{1}}c_{\nu'*\tau_{1}}^{r}p_{\tau}c_{\tau}^{r}p_{\tau_{\left|\tau\right|}*\nu''}c_{\tau_{\left|\tau\right|}*\nu''}^{r}\right)^{s_{r}/\left(s_{r}+r\right)}\\
\leq & ~\sum_{H\in\mathcal{M}}\sum_{\nu',\nu''\in F^{*}}\left(p_{\nu'}c_{\nu'}^{r}p_{\nu''}c_{\nu''}^{r}\right)^{\frac{s_{r}}{s_{r}+r}}\left(\sum_{\rho\in\Omega_{2}\cup\left\{ \theta\right\} }\left(p_{\rho}c_{\rho}^{r}\right)^{\frac{s_{r}}{s_{r}+r}}\right)^{2}\ \sum_{k=1}^{M_{3}}\sum_{\Gamma_{k}^{H,j}\left(\nu',\nu''\right)}\left(p_{\tau}c_{\tau}^{r}\right)^{\frac{s_{r}}{s_{r}+r}}\\
\le & ~\mbox{card}\left(\mathcal{M}\right)\left(\sum_{n=0}^{\infty}\sum_{\nu\in F_{n}}\left(p_{\nu}c_{\nu}^{r}\right)^{s_{r}/\left(s_{r}+r\right)}\right)^{2}\left(N^{2}+1\right)^{2}M_{3}M_{2}\ \ll\left(\sum_{n=0}^{\infty}t^{n}\right)^{2}\ll1.
\end{align*}
Combining this with Lemma \ref{lem2}, for $j\geq1$, we get 
\[
\phi_{j,r}\eta^{js_{r}/(s_{r}+r)}\asymp\sum_{\sigma\in\Lambda_{j,r}}\left(p_{\sigma}c_{\sigma}^{r}\right)^{s_{r}/(s_{r}+r)}\ll1
\]
and hence $\eta^{j}\ll\phi_{j,r}^{-(s_{r}+r)/s_{r}}$. This, together
with (\ref{characterization}), leads to 
\[
e_{\phi_{j,r}}^{r}(\mu)\asymp\sum_{\sigma\in\Lambda_{j,r}}p_{\sigma}c_{\sigma}^{r}\asymp\phi_{j,r}\cdot\eta^{j}\ll\phi_{j,r}\cdot\phi_{j,r}^{-(s_{r}+r)/s_{r}}=\phi_{j,r}^{-r/s_{r}}.
\]
It follows that $\phi_{j,r}^{r/s_{r}}e_{\phi_{j,r}}^{r}(\mu)\ll1$.
Hence, the assertion follows by Lemma \ref{g9}.

\end{proof}
In order to estimate the quantization error from below, we need an
auxiliary measure of Mauldin-Williams-type. One may see \citep[p. 823]{MW:88}
for more details.

Assume that, there are two elements $H_{1},H_{2}\in\mathcal{M}$ such
that $H_{1}\prec H_{2}$, i.e., there exists a path $\gamma=\left(i_{1},\ldots,i_{h}\right)$
satisfying 
\begin{eqnarray}
i_{1}\in H_{1},\;\; i_{2},\ldots,i_{h-1}\notin H_{1}\cup H_{2},\;\; i_{h}\in H_{2}.\label{g26}
\end{eqnarray}
Let $v=(v_{i})_{i=m_{1}+1}^{m_{2}}$ be the positive normalized right
eigenvector of $A_{H_{2},s_{r}}$ with respect to the Perron-Frobenius
eigenvector $1$. Set 
\begin{eqnarray*}
E_{q} & := & \{\tau\in H_{1}^{q}:\tau_{q}=1\},\;\widetilde{\gamma}:=\{i_{2},\ldots,i_{h-1}\},\\
F_{q} & := & \big\{\tau\ast\widetilde{\gamma}\ast\rho:\tau\in E_{q},\rho\in H_{2}^{\mathbb{N}},\rho_{1}=i_{h}\big\}.
\end{eqnarray*}
For every $\tau\in E_{q}$ and $\rho\in H_{2}^{*}(i_{h})$, we define
\begin{eqnarray}
\nu_{q}([\tau\ast\widetilde{\gamma}\ast\rho])=\left(p_{\tau}c_{\tau}^{r}p_{\widetilde{\gamma}}c_{\widetilde{\gamma}}^{r}p_{\rho}c_{\rho}^{r}\right)^{s_{r}/(s_{r}+r)}v_{\rho_{|\rho|}}.\label{s24}
\end{eqnarray}
where $[\tau\ast\widetilde{\gamma}\ast\rho]:=\{\tau\ast\widetilde{\gamma}\ast\omega:\omega\in(H_{2})_{\infty},\omega|_{|\rho|}=\rho\}$.
By (\ref{s24}), we have 
\begin{eqnarray*}
\sum_{i\in H_{2}}\nu_{q}([\tau\ast\widetilde{\gamma}\ast\rho\ast i]) & = & \sum_{i\in H_{2}}\left(p_{\tau}c_{\tau}^{r}p_{\widetilde{\gamma}}c_{\widetilde{\gamma}}^{r}\right)^{s_{r}/(s_{r}+r)}\left(p_{\rho\ast i}c_{\rho\ast i}^{r}\right)^{s_{r}/(s_{r}+r)}v_{i}\\
 & = & \left(p_{\tau}c_{\tau}^{r}p_{\widetilde{\gamma}}c_{\widetilde{\gamma}}^{r}p_{\rho}c_{\rho}^{r}\right)^{s_{r}/(s_{r}+r)}\sum_{i\in H_{2}}\left(p_{\rho_{|\rho|}i}c_{\rho_{|\rho|}i}^{r}\right)^{s_{r}/(s_{r}+r)}v_{i}\\
 & = & \left(p_{\tau}c_{\tau}^{r}p_{\widetilde{\gamma}}c_{\widetilde{\gamma}}^{r}p_{\rho}c_{\rho}^{r}\right)^{s_{r}/(s_{r}+r)}v_{\rho_{|\rho|}}.
\end{eqnarray*}
Thus, by Kolmogorov consistency theorem, we get a unique measure $\nu_{q}$
on $F_{q}$.
\begin{prop}
\label{mthm4} Assume that there are two comparable elements in $\mathcal{M}$.
Then we have, $\underline{Q}_{r}^{s_{r}}(\mu)=\infty$.\end{prop}
\begin{proof}
Assume that, there are two elements $H_{1},H_{2}\in\mathcal{M}$ such
that $H_{1}\prec H_{2}$. Without loss of generality, as in \citep{MW:88},
we assume that $H_{1}=\{1,\ldots,m_{1}\},i_{1}=1$ and $H_{2}=\{m_{2}+1,\ldots,m_{1}+m_{2}\},i_{h}=m_{2}+1$;
and (\ref{g26}) holds . As above, let $\widetilde{\gamma}:=(i_{2},\ldots,i_{h-1})$.
For all large $k$, there exist some words $\sigma\in\Lambda_{k,r}$
taking the form $\sigma=\tau\ast\widetilde{\gamma}\ast\rho$ with
$\tau\in E_{q}$ and $\rho\in H_{2}^{*},\rho_{1}=m_{2}+1$.

For every $q\leq l_{1k}-1-h$ and $\tau\in E_{q}$, we have, $p_{\tau\ast\widetilde{\gamma}}c_{\tau\ast\widetilde{\gamma}}^{r}\geq\eta^{-k}$,
otherwise, $\min_{\sigma\in\Lambda_{k,r}}|\sigma|$ would be strictly
less than $l_{1k}$, contradicting the definition of $l_{1k}$. This
implies that $\Lambda_{k,r}$ includes some subset $F_{q}^{\flat}$
of $F_{q}^{*}$ such that $\{\rho:\tau\ast\widetilde{\gamma}\ast\rho\in F_{q}^{\flat}\}$
forms a finite maximal antichain in $H_{2}^{*}(m_{2}+1):=\{\sigma\in H_{2}^{*}:\sigma_{1}=m_{2}+1\}$.
For each $\sigma=\tau\ast\widetilde{\gamma}\ast\rho\in F_{q}^{\flat}$,
we have 
\begin{eqnarray*}
\left(p_{\tau\ast\widetilde{\gamma}\ast\rho}c_{\tau\ast\widetilde{\gamma}\ast\rho}^{r}\right)^{s_{r}/(s_{r}+r)} & = & \left(p_{\tau}c_{\tau}^{r}p_{1i_{2}}c_{1i_{2}}^{r}p_{\widetilde{\gamma}}c_{\widetilde{\gamma}}^{r}p_{i_{h-1}\rho_{1}}c_{i_{h-1}\rho_{1}}^{r}p_{\rho}c_{\rho}^{r}\right)^{s_{r}/\left(s_{r}+r\right)}\\
 & \geq & \left(p_{\tau}c_{\tau}^{r}p_{\widetilde{\gamma}}c_{\widetilde{\gamma}}^{r}p_{\rho}c_{\rho}^{r}\right)^{s_{r}/(s_{r}+r)}\eta^{\frac{2s_{r}}{s_{r}+r}}\geq\eta^{\frac{2s_{r}}{s_{r}+r}}\nu([\sigma])\overline{v}^{-1}.
\end{eqnarray*}
Using this facts, we deduce 
\begin{eqnarray*}
 &  & \negmedspace\negmedspace\negmedspace\negmedspace\negmedspace\negmedspace\negmedspace\negmedspace\negmedspace\negmedspace\negmedspace\negmedspace\sum_{\sigma\in\Lambda_{k,r}}\left(p_{\sigma}c_{\sigma}^{r}\right)^{s_{r}/(s_{r}+r)}\\
 & \geq & \sum_{q=1}^{l_{1k}-1-h}\sum_{\sigma\in F_{q}^{\flat}}(p_{\sigma}c_{\sigma}^{r})^{s_{r}/(s_{r}+r)}\geq\overline{v}^{-1}\eta^{2s_{r}/(s_{r}+r)}\sum_{q=1}^{l_{1k}-1-h}\nu_{q}(F_{q})\\
 & = & \overline{v}^{-1}\eta^{2s_{r}/(s_{r}+r)}\left(p_{\widetilde{\gamma}}c_{\widetilde{\gamma}}\right)^{s_{r}/(s_{r}+r)}v_{m_{1}+1}\sum_{q=1}^{l_{1k}-1-h}\sum_{\tau\in E_{q}}\left(p_{\tau}c_{\tau}^{r}\right)^{s_{r}/(s_{r}+r)}=:Q_{k},
\end{eqnarray*}
where $v=(v_{i})_{i=1}^{m_{2}}$ is the positive eigenvector in the
definition of the measures $\nu_{q}$ and $\overline{v}:=\max_{1\leq i\leq m_{2}}v_{i}$.
Note that $s_{r}(H_{1})=s_{r}$. Let $w=(w_{i})$ be a positive left
eigenvector $A_{H_{1},s_{r}}$ with respect to the Perron-Frobenius
eigenvector $1$. Then we have $wA_{H_{1},s_{r}}^{h}=w$ for all $h\geq1$.
Let $A_{H_{1},s_{r}}^{q-1}=(c_{ij})_{m_{1}\times m_{1}}$. We have
\[
\sum_{\tau\in E_{q}}\left(p_{\tau}c_{\tau}^{r}\right)^{s_{r}/\left(s_{r}+r\right)}=\sum_{i=1}^{m_{1}}c_{i1}\geq\overline{w}^{-1}w_{1}.
\]
This implies that $Q_{k}\to\infty$. Thus, by (\ref{characterization})
and H\"older's inequality, we have 
\[
e_{\phi_{k,r},r}^{r}(\mu)\geq D\bigg(\sum_{\sigma\in\Lambda_{k,r}}\left(p_{\sigma}c_{\sigma}^{r}\right)^{s_{r}/\left(s_{r}+r\right)}\bigg)^{\left(s_{r}+r\right)/s_{r}}\phi_{k,r}^{-r/s_{r}}\geq Q_{k}^{\left(s_{r}+r\right)/s_{r}}\phi_{k,r}^{-r/s_{r}}.
\]
Hence, by Lemma \ref{g9}, it follows that $\underline{Q}_{r}^{s_{r}}(\mu)=\infty$.
The proposition follows.
\end{proof}
\begin{proof}
[Proof of Theorem \ref{mthm1}] For the proof of Theorem \ref{mthm1}
we just have to combine Proposition \ref{mthm3} and \ref{mthm4}.
\end{proof}
Next, we construct two examples illustrating Theorem \ref{mthm1}.
\begin{example}
Let $Q=(q_{ij})_{2\times2},T=(t_{ij})_{3\times3}$ be two positive
matrices, i.e., $q_{ij}>0,1\leq i,j\leq2$ and $t_{ij}>0,1\leq i,j\leq3$.
We define 
\[
P=\left(\begin{array}{cc}
Q_{2\times2} & 0\\
0 & T_{3\times3}
\end{array}\right).
\]
Then $P$ is a reducible matrix. Let $\mu$ be the Markov-type measure
associated with $P$. Let $H_{1}:=\{1,2\}$ and $H_{2}:=\{3,4\}$.
Clearly, $\mathcal{M}=\{H_{1},H_{2}\}$ and $H_{1},H_{2}$ are incomparable.
Thus, by Theorem \ref{mthm3}, $0<\underline{Q}_{r}^{s_{r}}(\mu)\leq\overline{Q}_{r}^{s_{r}}(\mu)<\infty$. 
\end{example}
\begin{example}
Let the transition matrix be given by 
\[
P=(p_{ij})_{4\times4}=\left(\begin{array}{cccc}
1/4 & 1/4 & 1/2 & 0\\
1/4 & 1/4 & 1/2 & 0\\
0 & 0 & 1/2 & 1/2\\
0 & 0 & 1/2 & 1/2
\end{array}\right).
\]
Fix $r>0$ we set $s:=r/(2r+1)$ implying $2\left(4^{-1}4^{-r}\right)^{s/\left(s+r\right)}=1.$
Let $c_{i,j}=1/8$ for $1\leq i\leq2$, $1\leq j\leq3$, and $c_{33}=c_{34}=c_{43}=c_{44}=2^{-1/s}.$
With $H_{1}:=\{1,2\}$ and $H_{2}:=\{3,4\}$ we have 
\[
A_{H_{1},s}=A_{H_{2},s}=\begin{pmatrix}\left(2^{-(2+2r)}\right)^{s/\left(s+r\right)} & \left(2^{-(2+2r)}\right)^{s/\left(s+r\right)}\\
\left(2^{-(2+2r)}\right)^{s/\left(s+r\right)} & \left(2^{-(2+2r)}\right)^{s/\left(s+r\right)}
\end{pmatrix}.
\]
Clearly $A_{H_{1},s}$, $A_{H_{2},s}$ are irreducible row-stochastic
matrices. Hence, $s_{r}=s=s_{r}(H_{i})$, $i=1,2$. Since $H_{1}\prec H_{2}$,
by Theorem \ref{mthm1}, we conclude that $\underline{Q}_{r}^{s_{r}}(\mu)=\infty$. \end{example}

\end{document}